\documentclass[a4paper]{article}

\usepackage[utf8]{inputenc}
\usepackage{lmodern}
\usepackage{amssymb,amsfonts,amsthm,amsmath,bbm,mathrsfs,aliascnt,mathtools,pifont}
\usepackage[english]{babel}
\usepackage[colorlinks]{hyperref}
\usepackage{tikz}
\usetikzlibrary{decorations}
\usetikzlibrary{decorations.pathreplacing}
\tikzset{individu/.style={draw,thick}}

\newcounter{dummy} \numberwithin{dummy}{section}
\newtheorem{theorem}[dummy]{Theorem}
\newtheorem{lemma}[dummy]{Lemma}
\newtheorem{proposition}[dummy]{Proposition}

\newtheorem{remark}[dummy]{Remark}

\usepackage{bm} 
\usepackage{amsmath}
\usepackage{graphicx}

\usepackage{xcolor}
\usepackage{hyperref} 
\hypersetup{backref=true,       
    pagebackref=true,               
    hyperindex=true,                
    colorlinks=true,                
    breaklinks=true,                
    urlcolor= black,                
    linkcolor= blue,                
    bookmarks=true,                 
    bookmarksopen=false,
    filecolor=black,
    citecolor=green,
    linkbordercolor=blue
}
\def\MR#1{\href{http://www.ams.org/mathscinet-getitem?mr=#1}{MR#1}}

\usepackage{amsthm}
\usepackage{amsmath}
\usepackage{amsmath}
\usepackage{amssymb}
\usepackage{mathrsfs}
\usepackage{graphicx}
\usepackage{mwe}
\usepackage{subfig}
\usepackage{caption}
\usepackage{subcaption}
\usepackage{ucs}
\usepackage{yfonts}
\usepackage{bbm}
\usepackage{graphicx}
\usepackage{enumerate}
\usepackage{xifthen}
\usepackage{upgreek} 
\usepackage{stmaryrd} 
\usepackage{xcolor}
\usepackage{mathrsfs}
\usepackage{bm} 
\usepackage{mathabx} 
\usepackage{geometry}
\usepackage{indentfirst}
\usepackage{enumitem}
\usepackage{hyperref}

\usepackage{tcolorbox} 

\numberwithin{equation}{section}

\makeatletter \newcommand\mathof[1]{{\operator@font#1}} \makeatother

\newcommand*\diff{\mathop{}\!\mathrm{d}}

\title{How long does it take to train an Elephant \\ Random Walk}
\author{Zheng Fang\thanks{Institute of Mathematics, University of Zurich, Zurich, Switzerland. \href{mailto: zheng.fang@math.uzh.ch}{zheng.fang@math.uzh.ch}}}
\vspace{-1cm}
\date{}
\begin{document}

\maketitle

\begin{abstract}
   We study how conditioning on the first $k$ steps, which we think of as training, affects the long-term behavior of the Elephant Random Walk. When the elephant is conditioned to be at position $k$ at time $k$, the first return time to the origin scales as $k^{(4-4p)/(3-4p)}$ in the diffusive regime, and grows exponentially in the critical regime. We loosely interpret this as a measurement of the rate at which the elephant forgets its training.
\end{abstract}

\section{Introduction} \label{section:introduction}
\subsection{Motivation}\label{subsection: motivation}
Imagine at a circus, we watch as an elephant performs a series of impressive tricks under the guidance of its mahout, a person who controls and cares for the elephant. This naturally leads us to wonder: how long does it take for the elephant to be trained so that its behavior becomes distinguishable from that of an untrained one? Building on this idea of training and memory, we study the Elephant Random Walk (ERW), a memory-dependent random walk introduced by Schütz and Trimper in \cite{elephantsalwaysremember}. Our main purpose is to gain some insights into how training changes its behavior and makes it possible to distinguish trained walks from untrained ones. In what follows, we say that an elephant is \textit{trained} once the effect of training becomes observable in its long-time behavior.

The ERW is a one-dimensional discrete-time nearest-neighbor process taking values in $\mathbb{Z}$ with complete memory of its past. For the \textit{untrained} ERW $(S(n))_{n \geq 0}$, we set $S(0) = 0$, and define $S(n) := \sum_{i=1}^{n} X_i$, where each $X_i \in \{-1, 1\}$ is the $i$-th step. The walk is driven by a memory parameter $p \in [0, 1]$. For $n \geq 2$, a past time $u(n) \in \{1, \ldots, n-1\}$ is chosen uniformly at random. With probability $p$, the walker repeats the step at the chosen time $u(n)$; otherwise, it repeats the opposite direction. When $p = 1/2$, the process is a simple random walk. We say the elephant is trained for $k$ steps, which will be denoted as $(S^{(k)}(n))_{n\geq 0}$, when the first $k$ steps $X_1, \ldots, X_k$ are fixed, the case we often consider hereafter is $X_1 = \ldots = X_k = 1$. And we are naturally interested in regime where $k,n\rightarrow\infty$ with $k\ll n$. That is, the duration of training is small compared to the time at which we observe the elephant. 

It is well known that the asymptotic behavior of the ERW, both trained (with a fixed number of initial steps) and untrained, exhibits a phase transition at $p=3/4$. For $p<3/4$, the walk is in the diffusive regime (scaling exponent below $1/2$), while for $p>3/4$, it is in the superdiffusive regime (scaling exponent $2p-1>1/2$). It has been stressed in \cite{connectiontourns} that this phase transition mirrors that for a two-color P\'olya-type urn with mean replacement matrix
\begin{equation*}
    R=
\begin{pmatrix}
    p & 1-p\\
    1-p & p
\end{pmatrix}
\end{equation*}
which has two eigenvalues $\lambda_1=1$ and $\lambda_2=2p-1$. It is well-known that the asymptotic behavior of such urn model depends on whether the second eigenvalue of the mean replacement matrix is lower or higher than half of the first eigenvalue, which points to the phase transition for the ERW.

The asymptotic behavior of the ERW in different regimes is well-known. For detailed analysis, see \cite{{connectiontourns},  {martingaleapproach}, {universality}, {CLT}, {invariance}, {recurrencefor1dERW}, {fixpoint}, {superdiffusivelaw}, {correlatedBernoulli}, {superdiffusiveERW}, {LucilePhDthesis}}. In the superdiffusive regime (for $p>3/4$), we have 
\begin{equation*}\label{intro: superdiffusive result}
    \frac{S(n)}{n^{2p-1}}\xrightarrow[n\rightarrow\infty]{a.s} L
\end{equation*}
where $L$ is a non-degenerate random variable, which has mean zero when there has been no training. However, even if we only train the elephant to move into $+1$ in the first step, denoting this elephant by $S^{(1)}$, then the corresponding limiting random variable $L^{(1)}$ has a positive mean. This shows that the superdiffusive ERW is trained already after the first step. In the diffusive regime (i.e., $p<3/4$), when the training is absent, the same rescaling applies for every memory parameter $p<3/4$: the rescaled process $\bigl(n^{-1/2} S(\lfloor nt \rfloor)\bigr)_{t \geq 0}$ converges in distribution as $n \to \infty$, in the sense of Skorokhod. More specifically, we have:
\begin{equation}\label{intro diffusive ERW}
    \left(\frac{S(\lfloor nt \rfloor)}{\sqrt{n}}\right)_{t\geq 0}  \xrightarrow[n\rightarrow\infty]{(d)} \left(\hat{B}(t)\right)_{t\geq 0}
\end{equation}
where the limiting process on the right-hand side is the so-called noise-reinforced Brownian motion with memory parameter $p$, a centered Gaussian process whose covariance function is given by
\begin{equation*}
    \mathbb{E}(\hat{B}(t)\hat{B}(s))=t^{2p-1}s^{2-2p}/(3-4p) \quad \text{$\forall 0 \leq s \leq t$}
\end{equation*}

On the other hand, the spectral gap plays an important role in the mixing properties in the setting of the Markov chains on a finite space. Informally, the larger the spectral gap is, the faster the Markov chain "forgets" its initial value; see e.g. {\cite[Chapter 12.2]{Levin-Peres}}. It is natural to expect that, for the ERW, a smaller memory parameter $p$ corresponds to a larger spectral gap $2-2p$, meaning the elephant forgets its initial training more quickly. In particular, this suggests that in the diffusive regime, the spectral gap should play a more visible role in determining the asymptotic behavior. The motivation of the current work is to make this intuition rigorous and perform a quantitative study.

The present work is also motivated by \cite{ChrisDean}, where functional limit theorems are established for balanced Pólya urns in regimes where both the number of draws and the initial composition tend to infinity. Through the standard correspondence between the elephant random walk and a two-color Pólya-type urn, the trained ERW considered here falls within the so-called Time-Step Dominant Small Urns regime in \cite{ChrisDean}. Accordingly, Proposition \ref{functional convergence for rescaled and conditioned ERW} should be understood primarily as an ERW-specific reformulation of \cite[Theorem3.1]{ChrisDean}. Likewise, the martingale constructions used in our proofs rely on techniques that are by now standard in the ERW literature. The genuinely new contribution of the present work is the extraction of the asymptotic behavior of the first return time to the origin for the trained ERW under growing initial conditioning. More precisely, we identify the critical training scales in the diffusive and critical regimes and derive the corresponding scaling limits and limit laws for this stopping time. In this way, the paper complements the results of Dean and extends the return-to-zero results of \cite{countingzeros} and \cite{criticalERW}, which deal with the first-return problem for the untrained ERW.


\begin{figure}
  \centering
  \includegraphics[scale = 0.55]{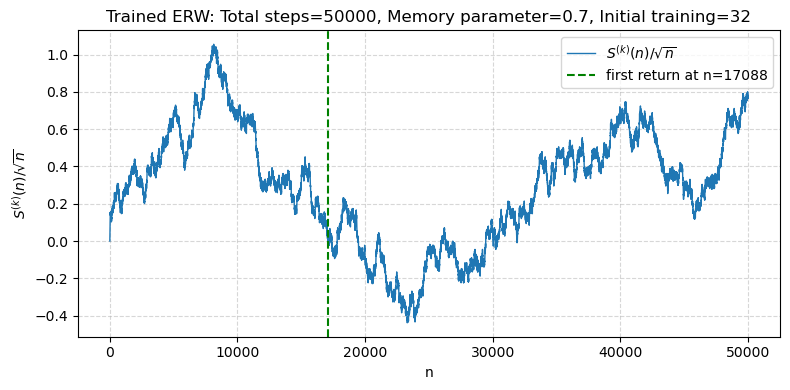}       
  \vspace{0.15cm}
  \includegraphics[scale = 0.55]{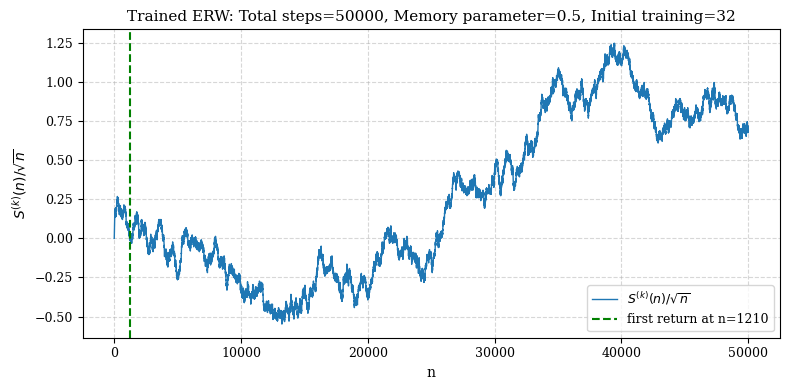}    
  \vspace{0.15cm}
  \includegraphics[scale = 0.55]{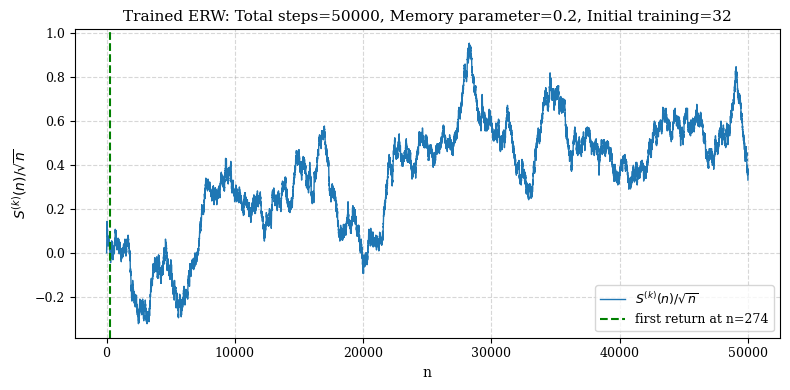}     
  \caption{We now present simulation results for the trained ERW in the diffusive regime after rescaling. Each plot shows a trajectory of $50,000$ steps with memory parameters $p = 0.7$, $0.5$, and $0.2$, and an initial training period fixed at $k=32$. In the first figure, we observe a clear long excursion away from the origin before the docile elephant (when $p>0.5$) returns for the first time (marked by the green dotted line). This excursion is much longer than $32^2 = 1024$, the asymptotic first return time for $p = 0.5$, illustrated in the middle plot for the classical simple random walk case. In contrast, in the last figure, the rebellious elephant (when $p<0.5$) returns even quicker as expected.}
  \label{figure 1}
\end{figure}

In the diffusive regime, we argue that the \textit{critical training phase} for the elephant occurs when $k = k(n) \sim n^{(3-4p)/(4-4p)}$. In other words, one must train the elephant for roughly $n^{(3-4p)/(4-4p)}$ steps to observe noticeable differences from an untrained elephant at time $n$. Whereas, if $k(n)\gg n^{(3-4p)/(4-4p)}$, the effect of training becomes evident. Here is a more precise statement.

\begin{proposition}\label{functional convergence for rescaled and conditioned ERW}
    In the diffusive regime, $p\in[0,3/4)$, if 
    \begin{equation}\label{critical training phase}
        k(n)\sim (3-4p)^{-1/(4-4p)}n^{(3-4p)/(4-4p)} \; \text{as $n\rightarrow\infty$}
    \end{equation}
    then the following weak convergence holds in the Skorokhod topology: for each pair of $0<\epsilon\leq l <\infty$,
\begin{equation}\label{ERW scaling limit diffusive 1}
    \left(\frac{S^{(k(n))}\left(\lfloor{nt}\rfloor\right)}{\sqrt{n}}\right)_{t\in[\epsilon,l] }\xrightarrow[n\rightarrow\infty]{(d)} \left(\hat{B}(t)+\frac{t^{2p-1}}{\sqrt{3-4p}}\right)_{t\in[\epsilon,l] }
\end{equation}
where $(\hat{B}(t))_{t\geq 0}$ is the noise-reinforced Brownian motion with parameter $p$. We emphasize that the convergence in (\ref{ERW scaling limit diffusive 1}) is valid for any positive compact interval bounded away from $0$.
\end{proposition}

We now dwell on a simple observation about the right-hand side of (\ref{ERW scaling limit diffusive 1}), which is at the origin of this study. The additional correction term of order $t^{2p-1}$ in (\ref{ERW scaling limit diffusive 1}), compared to (\ref{intro diffusive ERW}), comes from the training we imposed on the elephant. We note that if $p<1/2$, this correction term vanishes as $t$ grows larger. We refer to such an elephant as rebellious, as it quickly forgets its initial training. At the opposite for $p>1/2$, the correction term diverges as $t\rightarrow\infty$, but remains negligible compared to the standard deviation of the noise-reinforced Brownian motion. In this case, the elephant retains the influence of its training for a long time, and we call it docile. Since $\hat{B}$ is recurrent at the origin, the previous argument implies that, for every $p \in (0, 3/4)$, the limiting process in (\ref{ERW scaling limit diffusive 1}) also remains recurrent at the origin. 

Now recall, the law of the iterated logarithm for the noise-reinforced Brownian motion (see Equation (4) in \cite{universality})

\begin{equation*}\label{E:LIL} 
\limsup_{t\to 0+} \frac{|\hat B(t)|}{\sqrt{2t \ln \ln (1/t)}} = \frac{1}{\sqrt{3-4p}} \qquad \text{a.s.}
\end{equation*}
shows that for small times, the noise-reinforced Brownian motion is dominated by the correction term. 

We define an excursion away from the origin as a segment of the path between two successive visits to zero. With this definition, the continuity of the limiting process in (\ref{ERW scaling limit diffusive 1}), combined with the law of the iterated logarithm for the noise-reinforced Brownian motion, implies that when $p \in (1/2, 3/4)$, the limiting process in (\ref{ERW scaling limit diffusive 1}) immediately makes a first \textbf{excursion} away from the origin of strictly positive length. This motivates our study of its asymptotic behaviour in Theorem \ref{first return time for original ERW}. In contrast, when $p<1/2$, the vanishing correction term leads to such limiting process having an instantaneous return to the origin from infinity. We also notice the previous conclusions fit the trend described in Figure \ref{figure 1}.

Moreover, when the initial training time $k=k(n)$ is negligible relative to the critical training phase, namely, $k(n) \ll  n^{(3-4p)/(4-4p)}$, the correction term vanishes in the limit and the process reduces to that of (\ref{intro diffusive ERW}). 

\begin{remark}
We also note that when $p\in (1/2, 3/4)$, the limiting process on the right-hand side of (\ref{ERW scaling limit diffusive 1}) satisfies the stochastic differential equation,
\begin{equation} \label{intro sde}                    
\diff X(t)= \diff B(t) + \frac{2p-1}{t} X(t) \diff t \qquad \text{ with\, } X(0)=0 
\end{equation}
indicating that the solution to this SDE is not unique, since $\hat{B}$ is also a solution to (\ref{intro sde}) (see \cite{universality} for further details). 
\end{remark}

\subsection{Main Result}


The main purpose of this work is to investigate the asymptotic behavior of the first return time to the origin, $T^{(k)}$, for the trained ERW, $(S^{(k)}(n))_{n\geq 0}$, which is one way of measuring the persistence of the memory of the training. Of course, the elephant does not suddenly lose its entire memory at time $T^{(k)}$ (for instance, the first step is always recalled infinitely often). However, the steps $k+1,\, k+2,\, \dots$ after the initial training are heavily unbalanced \footnote{Here "unbalanced" means that the conditional distribution of the next step is asymmetric, i.e., $\mathbb{P}(X^{(k)}_{n+1}=\pm 1\mid\mathcal F_n)\neq 1/2$.}, whereas $T^{(k)}$ is the first time at which the elephant has the same probability $1/2$ of making an upward or downward step. More precisely, for every $k\in\mathbb{N}$, this return time is defined as
\begin{equation}\label{definition of return time diffusive}
    T^{(k)}:=\inf\{n\geq k: S^{(k)}(n)=0\}
\end{equation}


\begin{remark}
   Another natural quantity for measuring the persistence of the initial training is the asymptotic growth rate of the conditional displacement of the elephant, namely
   \begin{equation*}
   m_n:=\mathbb{E}\left(S^{(k)}(n)\mid S^{(k)}(k)=k\right), \quad n\geq k
   \end{equation*}
   Using
   \begin{equation*}
   \mathbb{E}\left(X^{(k)}_{n+1}\mid \mathcal{F}_n\right)= \frac{2p-1}{n}S^{(k)}(n)
   \end{equation*}
   we obtain
   \begin{equation*}
   \mathbb{E}\left(S^{(k)}(n+1)\mid \mathcal{F}_n\right)=\left(1+\frac{2p-1}{n}\right)S^{(k)}(n)
   \end{equation*}
   Taking the expectation yields the recursion formula
   \begin{equation*}
   m_{n+1}=\left(1+\frac{2p-1}{n}\right)m_n, \quad \text{with } m_k=k
   \end{equation*}
   Therefore, together with Stirling's formula
   \begin{equation*}
   \mathbb{E}\left(S^{(k)}(n)\mid S^{(k)}(k)=k\right)=k\prod_{j=k}^{n-1}\left(1+\frac{2p-1}{j}\right)\sim k^{2-2p}n^{2p-1}\quad \text{as } n\to\infty
   \end{equation*}
   In particular, this indicates that the effect of the initial training fades when $p<1/2$, as the above conditional displacement tends to $0$ as $n\to\infty$. By contrast, the effect of the training persists when $p>1/2$, since the conditional displacement diverges as $n\to\infty$.
   
   We stress that this point of view is much less precise than the approach based on the first return time. It does not determine when the effect of the initial training is exactly cancelled. Instead, it rather captures the rate at which the influence of the training is balanced on average.
\end{remark}

In the diffusive regime, the rescaling given in (\ref{ERW scaling limit diffusive 1}) suggests that to obtain the proper scaling for $T^{(k)}$, we should also consider such a time-reparametrized process. We emphasize that we are interested in the asymptotic long-term behavior of the trained ERW process, where the initial training period $k=k(n)$ is negligible compared to the observation time $n$, namely $k(n)\ll n$.

Followed by the idea presented above, (\ref{definition of return time diffusive}) can be written as
\begin{equation}\label{definition of return time of rescaled ERW}
   T^{(k(n))}=\lfloor{k(n)}\rfloor+\inf\left\{m\geq 0: S^{(k(n))}\left(\lfloor{k(n)}\rfloor+m\right)=0\right\}
\end{equation}
We adopt the time reparametrization suggested by Proposition \ref{functional convergence for rescaled and conditioned ERW}, which allows us to write (\ref{definition of return time of rescaled ERW}) as the following:
\begin{equation*}
    T^{(k(n))}=\lfloor{k(n)}\rfloor+n\cdot\inf\left\{t\geq 0: S^{(k(n))}\left(\lfloor{k(n)}\rfloor+\lfloor{nt}\rfloor\right)=0\right\} 
\end{equation*}
 Furthermore, we will argue that for large $n$, the process $\left(S^{(k(n))}\left(\lfloor{nt}\rfloor\right)\right)_{t\geq 0}$ is unlikely to return to the origin on a time scale smaller than order $n$, i.e., it exhibits no fast return. After a change of variables, this leads to the next result.

\begin{theorem} \label{first return time for original ERW}
In the diffusive regime, $p\in[0, 3/4)$, we have
\begin{equation}\label{intro: return time diffusive}
   k^{-\frac{4-4p}{3-4p}}T^{(k)} \xrightarrow[k\rightarrow\infty]{(d)} (3-4p)^{\frac{1}{3-4p}}\tau^{\frac{1}{3-4p}}
\end{equation}
where $\tau$ is a $\text{Stable}(1/2)$-distributed random variable with density
\begin{equation*}
    \frac{1}{\sqrt{2\pi t^{3}}}\exp{\left(-\frac{1}{2t}\right)} \quad \text{for $t>0$.}
\end{equation*} 
\end{theorem}

The exponent $(4-4p)/(3-4p)$ in (\ref{intro: return time diffusive}) is greater than $2$ for the elephant with memory parameter $p>1/2$, whereas it is smaller than $2$ for $p<1/2$. For $p=1/2$, it reduces to the simple random walk and recovers the classical result: the first return time to the origin of the walker starting at position $k$ scales as $k^2$. The simulations in Figure \ref{figure 1} are consistent with Theorem \ref{first return time for original ERW}, confirming that the ERW with memory parameter $p>1/2$ indeed takes longer to return to the origin, i.e., to forget the initial training. Below, we also provide a numerical illustration of Theorem \ref{first return time for original ERW} in Figure \ref{figure 2}.

\begin{figure}
  \centering
  \includegraphics[scale = 0.55]{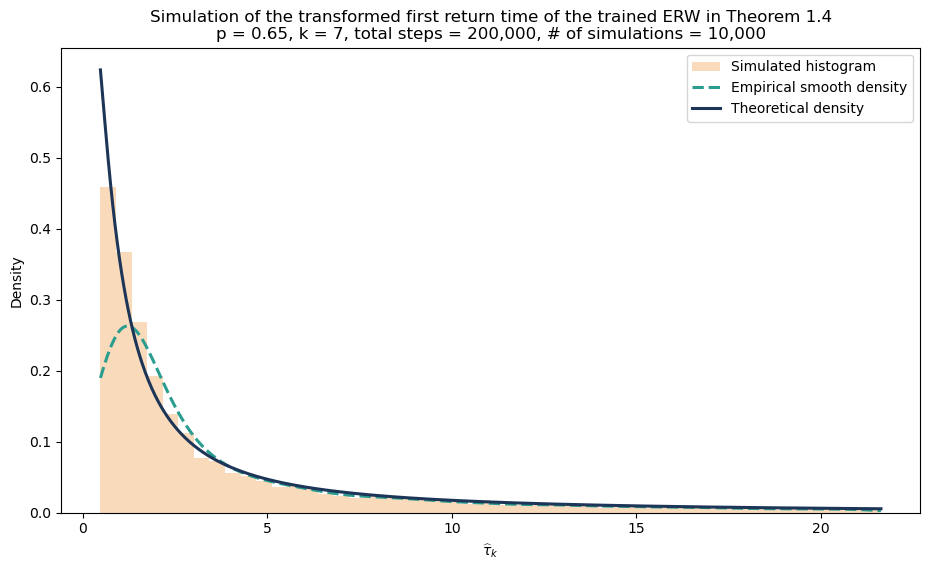}           
  \caption{This plot numerically illustrates Theorem \ref{first return time for original ERW}. According to the theorem, after the transformation, $\hat{\tau}_k:=(T^{(k)}/((3-4p)^{\frac{1}{3-4p}}k^{\frac{4-4p}{3-4p}}))^{3-4p}$, the first return time $T^{(k)}$ is expected to converge in distribution to a Stable($1/2$) law and the solid curve shows the corresponding theoretical density. The histogram shows the simulated distribution of this transformed quantity $\hat{\tau}_k$. The good agreement between the two confirms the scaling and limiting behavior predicted by Theorem \ref{first return time for original ERW}.}
  \label{figure 2}
\end{figure}

Our results also have a natural connection with persistence, but our setting
should be distinguished from the standard one. In the persistence literature
(see, for example, \cite{Persistenceprobabilitiesandexponents,Persistenceandfirst-passageproperties}), one usually fixes the initial condition and studies the large-time decay of a survival probability. More precisely, if $Y$ is a process started above zero and
\begin{equation*}
    T_0:=\inf\{t\geq 0:Y(t)\leq 0\}
\end{equation*}
denote its first-passage time to $0$. Then the persistence exponent $\theta$ is defined through the tail behavior of the first-passage time 
\begin{equation*}
    \mathbb P(T_0>t)\sim t^{-\theta},\quad t\to\infty
\end{equation*}
For one-dimensional Brownian motion and simple symmetric random walks started
from a fixed positive position, the corresponding exponent is $\theta=1/2$.

The setting of the present paper is different. Here the memory parameter $p$ and the transition mechanism are fixed, but the initial condition depends on the training length $k$. Thus the initial displacement after training grows with $k$. The exponent below therefore describes the tail behavior of the limiting rescaled return time to $0$ with this growing initial displacement.

More precisely, Theorem \ref{first return time for original ERW} shows that in the diffusive regime 
\begin{equation*}
    k^{-\frac{4-4p}{3-4p}}T^{(k)} \xrightarrow[k\rightarrow\infty]{(d)} (3-4p)^{\frac{1}{3-4p}}\tau^{\frac{1}{3-4p}}
\end{equation*}
where $\tau=\inf\{u\geq 0:B(u)+1=0\}$ is the first hitting time of $0$ by a Brownian motion started from $1$. Since $\mathbb{P}(\tau>u)\sim cu^{-1/2}$, the limiting rescaled return-time has tail distribution
\begin{equation*}
    \mathbb{P}\left((3-4p)^{1/(3-4p)}\tau^{1/(3-4p)}>t\right)\sim c_pt^{-(3-4p)/2} \;\text{ as $t\to\infty$}
\end{equation*}
Thus, the limiting rescaled return time to $0$ has tail exponent $(3-4p)/2$, and it can be viewed as the Brownian persistence exponent $1/2$ after the nonlinear time change $u=t^{3-4p}$. In particular, for $p>1/2$ this exponent is smaller than $1/2$, so the limiting tail decays more slowly, corresponding to a strong persistence (the docile elephant); while for $p<1/2$, the exponent is larger than $1/2$, so the limiting tail decays faster, corresponding to weak persistence (the rebellious elephant). At $p=1/2$, one recovers the classical exponent $1/2$.

We also note that the quickest return to the origin for the conditioned ERW occurs at $p=0$, nevertheless, the return time in this case still scales as $k^{4/3}$ and not as $k$ as one might naively expect. This highlights the impact of the training on the return dynamics. In contrast to the case here with initial training, the behavior of the first return time for the untrained ERW is thoroughly analyzed in \cite{countingzeros} for the diffusive regime and in \cite{criticalERW} for the critical regime.

Later, we extend Theorem \ref{first return time for original ERW} to the case where the initial training brings the elephant to a position of order $ck$ at time $k$, with $c\in(0,1]$. By symmetry, the case $c\in[-1,0)$ is equivalent, so we restrict attention to $c\in(0,1]$, focusing primarily on the extreme case $c=1$. This setting is also similar to the framework considered by C. Dean in \cite{ChrisDean}. We stress that the dynamics for $c=0$ is excluded from our discussion, as it likely requires new ideas to handle.

The structure of the paper is as follows. In Section \ref{section:preliminaries}, inspired by the martingale approach in \cite{invariance}, we introduce properties of the trained ERW in the diffusive regime and establish Proposition \ref{functional convergence for rescaled and conditioned ERW}, a functional limit theorem for the rescaled trained ERW. Section \ref{main proof} is devoted to proving, Theorem \ref{first return time for original ERW}, one of our main results on the asymptotic behavior of its first return time to the origin, followed by a generalization to ERWs with a broader class of initial training. Finally, in Section \ref{critical regime}, we state analogous results for the critical regime, $p=3/4$. In this text, we use $S^{(k)}$ for the ERW trained for $k$ steps, and $S^{(k(n))}$ for the ERW trained for $\lfloor{k(n)}\rfloor$ steps. And we write $f(n)\sim g(n)$ as $n\rightarrow\infty$ if we have $\lim_{n\rightarrow\infty}f(n)/g(n)=1$. We write $f_n\sim_{\mathbb{P}} g_n$ as $n\rightarrow\infty$ to indicate that $f_n-g_n\xrightarrow[n\rightarrow\infty]{(\mathbb{P})}0$, that is, the difference converges to zero in probability.

To conclude this introduction, we point out that it would be interesting to investigate similar questions in dimension 2, where the recurrence of the 2-dimensional ERW has been studied by Qin \cite{ShuoQin} and by Curien \& Laulin \cite{NicolasCurien}. The planar case, however, is much more challenging and will likely require new ideas.

\section{First results for the trained ERW} \label{section:preliminaries}

We recall that the trained ERW refers to the process $\left(S^{(k)}\left(n\right)\right)_{n\geq 0}$, where the elephant is trained for the first $k$ steps. In the diffusive regime, we adopt the time reparametrization $\left(S^{(k(n))}\left(\lfloor{nt}\rfloor\right)\right)_{t\geq 0}$. For each $n \in \mathbb{N}$, we denote by $\mathbb{P}_{k(n)}$ the law of the trained ERW, and by $\mathbb{E}_{k(n)}$ the corresponding expectation under $\mathbb{P}_{k(n)}$.


Recall in the diffusive regime, the memory parameter $p$ satisfies $p\in \left[0, 3/4\right)$. We set for $j\in\mathbb{N}$,
\begin{equation*}
    a_{j}:=\frac{\Gamma\left(j\right)}{\Gamma \left(j+2p-1 \right)}
\end{equation*}
Recall that, by Stirling’s formula,
\begin{equation}\label{asymp coeff}
    a_{j}\sim j^{1-2p} \quad \text{ as $j\rightarrow\infty$}
\end{equation}

We are then ready to define the process for $t\geq 0$
\begin{equation}\label{M_1^{(k,n)} martingale}
    \left(M_1^{(k,n)}(t)\right)_{t\geq 0}:= \left(a_{\lfloor{k(n)}\rfloor+\lfloor{nt}\rfloor}S^{(k(n))}\left(\lfloor{k(n)}\rfloor+\lfloor{nt}\rfloor\right)-a_{\lfloor{k(n)}\rfloor}S^{(k(n))}\left(\lfloor{k(n)}\rfloor\right)\right)_{t\geq 0}
\end{equation}

Though similar properties of the process $(M_1^{(k,n)}(t))_{t\geq 0}$ have been observed in other literatures, we include a proof here for completeness.

\begin{lemma}\label{diffusive martingale}
    We have the next properties for the process $(M_1^{(k,n)}(t))_{t\geq 0}$,
    \begin{enumerate}[label={\arabic*.}, ref=2.1.\arabic*]
        \item \label{part 1}  For every $n\in\mathbb{N}$, $(M_1^{(k,n)}(t))_{t\geq 0}$ is a martingale that starts from 0 under the law $\mathbb{P}_{k(n)}$.
        \item \label{part 2} For every $t\geq 0$, its quadratic variation has the following asymptotics
        \begin{equation*}\label{quadratic variation}
            \left[M_1^{(k,n)}, M_1^{(k,n)}\right](t) \sim \frac{t^{3-4p}}{3-4p}n^{3-4p} \quad \text{as $n\rightarrow\infty$}
        \end{equation*}
    \end{enumerate}
\end{lemma}

\begin{proof}
The proof of Lemma \ref{part 1} is straightforward, as the martingale property for the case $k(n)=0$ has already been shown in \cite{martingaleapproach, countingzeros, CLT}.  The extension to a general $k(n)\geq 1$ follows directly.  

Next, we turn to the proof of Lemma \ref{quadratic variation}. This result follows directly from the definition of the quadratic variation of a martingale. Specifically, we can write:

\begin{align*}
    \left[M_1^{(k,n)}, M_1^{(k,n)}\right](t)&=\sum_{s\leq t}\left(\Delta M_1^{(k,n)}(s)\right)^2\\ &=\int_{\lfloor{k(n)}\rfloor}^{\lfloor{k(n)}\rfloor+\lfloor{nt}\rfloor}\left(a_{\lfloor{u}\rfloor}S^{(k(n))}\left(\lfloor{u}\rfloor\right)-a_{\lfloor{u}\rfloor-1}S^{(k(n))}\left(\lfloor{u}\rfloor-1\right)\right)^2 \diff u\\
    &=\int_{\lfloor{k(n)}\rfloor}^{\lfloor{k(n)}\rfloor+\lfloor{nt}\rfloor}a_{\lfloor{u}\rfloor}^2\left(\pm1+(1-2p)\frac{S^{(k(n))}\left(\lfloor{u}\rfloor-1\right)}{\lfloor{u}\rfloor-1}\right)^2 \diff u \\
     & \sim \frac{t^{3-4p}}{3-4p}n^{3-4p} \quad \text{as $n\rightarrow\infty$}
\end{align*}
In the penultimate line, since $k(n)\ll n$, and by using an argument similar to the one in the proof of the law of large numbers for the untrained ERW (see {\cite[Corollary 1]{CLT}}), we get an analogous result for the trained elephant: for $u \in (\lfloor k(n)\rfloor, \lfloor k(n)\rfloor+\lfloor nt\rfloor)$, we have
\begin{equation*}
    \frac{S^{(k(n))}\left(\lfloor{u}\rfloor-1\right)}{\lfloor{u}\rfloor-1}=o(1) \quad \text{as $n\rightarrow\infty$}
\end{equation*}
To conclude the proof, we apply the asymptotics of $a_j$ given in (\ref{asymp coeff}).
\end{proof}

Before delving into the proof of Proposition \ref{functional convergence for rescaled and conditioned ERW}, we recall from Section \ref{subsection: motivation} that our result can be viewed as a reformulation of Dean’s result in {\cite[Theorem 3.1, Equation (3.6)]{ChrisDean}}, since the ERW model satisfies all the assumptions required for the Time-Step Dominant Small Urns regime. In particular, the rescaling appeared in Proposition \ref{functional convergence for rescaled and conditioned ERW} coincides with that of {\cite[Equation (3.6)]{ChrisDean}}. Nevertheless, our formulation is more precise in the context of the current model.

With the martingale theory tools in hand, we can apply the martingale functional central limit theorem (abbreviated as martingale FCLT from now on) to determine the asymptotic behavior of the conditioned ERW, which we will explore in the next section.

In this section, we aim to provide a detailed proof of the results concerning the long-time behavior of the conditioned ERW. As mentioned at the end of the previous section, our main tool is the martingale FCLT (e.g see {\cite[Theorem 2.1]{martingaleFCLT}}).

Motivated by the asymptotic behavior of quadratic variation provided in Lemma \ref{part 2}, we introduce the following process for each $n\in\mathbb{N}$,
\begin{equation}\label{N_1^{(n)}}
    N_{1}^{(k,n)}(t):=\frac{M_{1}^{(k,n)}(t)}{ \sqrt{\frac{n^{3-4p}}{3-4p}}}, \; \text{ for $t\geq 0$}
\end{equation}
 which is a martingale that starts from $0$.

We are now ready to apply the martingale FCLT, which implies that the process converges weakly to a Brownian motion in the Skorokhod topology. To ultimately deduce the scaling limit of $\left(S^{(k(n))}\left(\lfloor{k(n)}\rfloor+\lfloor{nt}\rfloor\right)\right)_{t\geq 0}$, from equations (\ref{M_1^{(k,n)} martingale}) and (\ref{N_1^{(n)}}), it is sufficient to require
\begin{equation*}
a_{\lfloor{k(n)}\rfloor}S^{(k(n))}\left(\lfloor{k(n)}\rfloor\right) \text{ grows at most at the rate of } \sqrt{\frac{n^{3-4p}}{3-4p}} 
\end{equation*}
Here, we focus on the non-degenerate case, meaning they grow at the same rate. In this scenario, we obtain the critical training phase
\begin{equation}\label{k(n) diffusive}
    k(n)\sim (3-4p)^{-1/(4-4p)} n^{(3-4p)/(4-4p)} \quad \text{ as $n\rightarrow\infty$}
\end{equation}
We stress here that $k(n)$ indeed grows slower than $n$ when $n$ is large enough.
\begin{proof}[Proof of Proposition \ref{functional convergence for rescaled and conditioned ERW}]
    We start off by showing the expected value of the maximum jump of the martingale process $\left(N_{1}^{(k,n)}(t)\right)_{t\geq 0}$ is asymptotically negligible as $n\rightarrow\infty$. First, we denote the martingale difference of $M_1^{(k,n)}$ at time $t$ by $\Delta M_1^{(k,n)}(t)$. Since the step size of the ERW is always $\pm 1$, we can write
    \begin{align*}
        \Delta M_1^{(k,n)}(t) &:= a_{\lfloor{k(n)}\rfloor+\lfloor{nt}\rfloor+1}S^{(k(n))}\left(\lfloor{k(n)}\rfloor+\lfloor{nt}\rfloor+1\right)-a_{\lfloor{k(n)}\rfloor+\lfloor{nt}\rfloor}S^{(k(n))}\left(\lfloor{k(n)}\rfloor+\lfloor{nt}\rfloor\right)\\
        &=\frac{1-2p}{\lfloor{k(n)}\rfloor+\lfloor{nt}\rfloor+2p-1}a_{\lfloor{k(n)}\rfloor+\lfloor{nt}\rfloor}S^{(k(n))}\left(\lfloor{k(n)}\rfloor+\lfloor{nt}\rfloor\right)\pm a_{\lfloor{k(n)}\rfloor+\lfloor{nt}\rfloor+1}
    \end{align*}
    and by observing the identity
    \begin{equation*}
        a_{\lfloor{k(n)}\rfloor+\lfloor{nt}\rfloor+1}-a_{\lfloor{k(n)}\rfloor+\lfloor{nt}\rfloor}=\frac{1-2p}{\lfloor{k(n)}\rfloor+\lfloor{nt}\rfloor+2p-1}a_{\lfloor{k(n)}\rfloor+\lfloor{nt}\rfloor}
    \end{equation*}
    together with $|S^{(k(n))}\left(\lfloor{k(n)}\rfloor+\lfloor{nt}\rfloor\right)|\leq \lfloor{k(n)}\rfloor+\lfloor{nt}\rfloor$ and $|1-2p|\leq 1$, we have the following upper bound for the martingale difference at time $t$,
    \begin{equation}\label{martingale difference diffusive}
        \left|\Delta M_1^{(k,n)}(t)\right|\leq 2 a_{\lfloor{k(n)}\rfloor+\lfloor{nt}\rfloor+1} \quad \text{$\mathbb{P}_{k(n)}$-a.s.}
    \end{equation}
    Furthermore, combining (\ref{asymp coeff}), (\ref{k(n) diffusive}), and the upper bound provided in (\ref{martingale difference diffusive}), we conclude that, for every $t\geq0$
    \begin{equation*}
        \mathbb{E}_{k(n)}\left(\Delta N_{1}^{(k,n)}(t)\right)=\mathbb{E}_{k(n)}\left(\frac{\Delta M_1^{(k,n)}(t)}{\sqrt{\frac{n^{3-4p}}{3-4p}}}\right)\xrightarrow{n\rightarrow\infty} 0
    \end{equation*}
On the other hand, directly from Lemma \ref{quadratic variation} and (\ref{N_1^{(n)}}), it follows for every $t\geq 0$,
    \begin{equation*}
        \left[N_{1}^{(k,n)},N_{1}^{(k,n)}\right](t)\xrightarrow[n\rightarrow\infty]{(d)} t^{3-4p} \quad \text{as $n\rightarrow\infty$}
    \end{equation*}
By martingale FCLT, we obtain the following weak convergence in the Skorokhod topology,
\begin{equation}\label{prelim convergence diffusive}
    \left(N_{1}^{(k,n)}(t)\right)_{t\geq 0}\xrightarrow[n\rightarrow\infty]{(d)} \left(B(t^{3-4p})\right)_{t\geq 0}
\end{equation}
Using the relation given in (\ref{k(n) diffusive}), we can equivalently rewrite (\ref{prelim convergence diffusive}) as the following convergence, which holds in the Skorokhod topology for each pair of $0<\epsilon\leq l <\infty$: 
\begin{equation}\label{ERW scaling limit diffusive}
    \left(\frac{S^{(k(n))}\left(\lfloor{k(n)}\rfloor+\lfloor{nt}\rfloor\right)}{\sqrt{n}}\right)_{t\in [\epsilon,l]}\xrightarrow[n\rightarrow\infty]{(d)} \left(\hat{B}(t)+\frac{t^{2p-1}}{\sqrt{3-4p}}\right)_{t\in [\epsilon,l]} 
\end{equation}
To show (\ref{ERW scaling limit diffusive 1}), we observe that 
\begin{equation}\label{equivalent result}
    \frac{S^{(k(n))}\left(\lfloor{nt}\rfloor\right)}{\sqrt{n}}\sim _{\mathbb{P}_{k(n)}}\frac{S^{(k(n))}\left(\lfloor{k(n)}\rfloor+\left\lfloor{n\left(t-\frac{k(n)}{n}\right)}\right\rfloor\right)}{\sqrt{n}} \quad\text{for large $n$}
\end{equation}
Since $k(n)\ll n$, by (\ref{ERW scaling limit diffusive}) and the continuity of $\hat{B}$, we recover the statement of Proposition \ref{functional convergence for rescaled and conditioned ERW}.
\end{proof}

Before proving Proposition \ref{convergence for k(n)>> critical training phase}, let us briefly explain the idea. In Proposition \ref{functional convergence for rescaled and conditioned ERW}, the training length $k(n)$ is tuned so that the deterministic contribution of the initial training and the martingale fluctuations are of the same order. Proposition \ref{convergence for k(n)>> critical training phase} treats the overtrained regime $k(n)\gg n^{(3-4p)/(4-4p)}$, where the deterministic effect of the initial training is larger than the critical training phase. Then the leading term is no longer the finite correction appearing in Proposition \ref{functional convergence for rescaled and conditioned ERW}, but the deterministic profile $k(n)^{2-2p}/(nt)^{1-2p}$. Therefore, the proof consists in going back to the martingale functional limit theorem (\ref{prelim convergence diffusive}), subtracting this leading term, and observing that the remaining fluctuations are still of order $\sqrt{n}$ and converge to a centered Gaussian limit. Evaluating the resulting process convergence at $t=1$ yields the statement. And here we emphasize that $k(n)$ still remains negligible compared to $n$.

\begin{proposition}\label{convergence for k(n)>> critical training phase}
    In the diffusive regime, $p\in[0,3/4)$, if $k(n) \gg n^{(3-4p)/(4-4p)}$ as $n\rightarrow\infty$, we have the following weak convergence
    \begin{align*}
        \frac{S^{(k(n))}\left(n\right)-\frac{k(n)^{2-2p}}{n^{1-2p}}}{\sqrt{n}}\xrightarrow[n\rightarrow\infty]{(d)}\mathcal{N}\left(0\, , \frac{1}{3-4p}\right)
    \end{align*}
\end{proposition}

\begin{proof}[Proof of Proposition \ref{convergence for k(n)>> critical training phase}]
The proof is a different reading of (\ref{prelim convergence diffusive}) in the regime where the deterministic contribution of the initial training dominates the critical training phase. Using (\ref{asymp coeff}) and the standing assumption $k(n)\ll n$, relation (\ref{prelim convergence diffusive}) gives
\begin{align*}
     \frac{(nt)^{1-2p}S^{(k(n))}\left(\lfloor{k(n)}\rfloor+\lfloor{nt}\rfloor\right)-k(n)^{2-2p}}{\sqrt{\frac{n^{3-4p}}{3-4p}}}\xrightarrow[n\rightarrow\infty]{(d)} B(t^{3-4p})
\end{align*}
After rearrangement, we have
\begin{equation*}
    \frac{S^{(k(n))}\left(\lfloor{k(n)}\rfloor+\lfloor{nt}\rfloor\right)-\frac{k(n)^{2-2p}}{(nt)^{1-2p}}}{\sqrt{n}}\xrightarrow[n\rightarrow\infty]{(d)} \frac{t^{2p-1}}{\sqrt{3-4p}}B(t^{3-4p})
\end{equation*}
After applying (\ref{equivalent result}) and setting $t=1$ yields the desired result.
\end{proof}

\section{Proof of Theorem \ref{first return time for original ERW}}\label{main proof}

The convergence in Proposition \ref{functional convergence for rescaled and conditioned ERW} holds on compact time intervals bounded away from $0$. By itself, it therefore does not exclude returns to the origin occurring at times that are negligible to $n$. To show that the first return time indeed occurs on the scale $n$, we need an additional estimate showing that returns on time scales smaller than $n$ are extremely unlikely.

\begin{lemma}\label{unlikely to return zero}
   Under the condition $k(n)\ll n$, the following uniform convergence holds for the process $\left(S^{(k(n))}\left(\lfloor k(n)\rfloor+\lfloor nt\rfloor\right)\right)_{t \geq 0}$ under $\mathbb{P}_{k(n)}$:
\begin{equation*}
   \sup_{n \in \mathbb{N}} \,
   \mathbb{P}_{k(n)}\!\left(
      \inf_{0 \leq t \leq \epsilon} 
      S^{(k(n))}\left(\lfloor k(n)\rfloor+\lfloor nt\rfloor\right) \leq 0
   \right) \xrightarrow{\epsilon\downarrow 0} 0.
\end{equation*}
\end{lemma}

\begin{proof}
   Fix $\epsilon>0$, and set
   \begin{equation*}
    A_{n,\epsilon}:= \left\{\inf_{0\leq t\leq \varepsilon}S^{(k(n))}\left(\lfloor{k(n)}\rfloor+\lfloor{nt}\rfloor\right)\leq 0\right\}
   \end{equation*}
   We first observe that $ A_{n,\epsilon}$ is empty whenever $\lfloor {n\epsilon}\rfloor < \lfloor {k(n)}\rfloor$. We now consider the nontrivial case. Since $a_j>0$ for every $j\in\mathbb{N}$, the definition (\ref{M_1^{(k,n)} martingale}) of $M_1^{(k,n)}$ yields
   \begin{align*}
   A_{n,\epsilon} &=\left\{\inf_{0\leq t\leq \epsilon} a_{\lfloor {k(n)}\rfloor+\lfloor {nt}\rfloor}S^{(k(n))}\left(\lfloor {k(n)}\rfloor+\lfloor {nt}\rfloor\right)\leq 0 \right\} =\left\{\sup_{0\leq t\leq \epsilon}\left(-M_1^{(k,n)}(t)\right)\geq a_{\lfloor{k(n)}\rfloor}\lfloor{k(n)}\rfloor\right\} 
   \end{align*}
   
    In this proof, we rely on the Azuma–Hoeffding inequality (e.g see {\cite[Chapter 3.1]{Concentrationinequalities}}). As an upper bound of the martingale difference of $M_1^{(k,n)}$ is given by (\ref{martingale difference diffusive}), the Azuma–Hoeffding inequality gives us for any $\epsilon>0$,
    \begin{align*}\label{Azuma diffusive}
        \mathbb{P}_{k(n)}\left(A_{n,\epsilon}\right)&\leq \exp\left(-\frac{1}{2}\frac{a_{\lfloor{k(n)}\rfloor}^2 \lfloor{k(n)}\rfloor^2}{\sum_{j=1}^{\lfloor{n\epsilon}\rfloor}\left(2a_{\lfloor{k(n)}\rfloor+j+1}\right)^2}\right)
    \end{align*}
    Together with (\ref{asymp coeff}) and (\ref{k(n) diffusive}), we conclude that for $n\rightarrow\infty$,
    \begin{equation*}
        \frac{a_{\lfloor{k(n)}\rfloor}^2 \lfloor{k(n)}\rfloor^2}{\sum_{j=1}^{\lfloor{n\epsilon}\rfloor}\left(a_{\lfloor{k(n)}\rfloor+j+1}\right)^2}\sim \frac{n^{3-4p}}{(n\epsilon)^{3-4p}}=\frac{1}{\epsilon^{3-4p}}
    \end{equation*}
    where we used $k(n)\ll n$. Thus, there exists $N_0\in\mathbb{N}$ such that for every $n\geq N_0$, we have 
    \begin{equation}\label{upper bound for Azuma}
        \mathbb{P}_{k(n)}\left(A_{n,\epsilon}\right)\leq \exp\left(-\frac{1}{4\epsilon^{3-4p}}\right)
    \end{equation}
    We emphasize that this choice of upper bound is independent of $n$, as shown in (\ref{upper bound for Azuma}).
    It remains to control the remaining finitely many integers $n<N_0$. Since this set is finite and $\lfloor{k(n)}\rfloor\geq 1$ for each $1\leq n<N_0$, the quantity
    \begin{equation*}
    \epsilon_0:=\frac{1}{2}\min_{1\leq n<N_0}\frac{\lfloor{k(n)}\rfloor}{n}
    \end{equation*}
    is strictly positive. Hence, whenever $0<\epsilon<\epsilon_0$, we have $\lfloor{ n\epsilon}\rfloor<\lfloor{k(n)}\rfloor\text{ for every }1\leq n<N_0$, and therefore the event $A_{n,\epsilon}$ is empty for all $n<N_0$.

   Together with (\ref{upper bound for Azuma}), we conclude that for every $0<\epsilon<\epsilon_0$ and all $n\in\mathbb{N}$,
   \begin{equation*}
   \sup_{n\in\mathbb N}\mathbb{P}_{k(n)}\left(A_{n,\epsilon}\right)\leq \exp\left(-\frac{1}{4\epsilon^{3-4p}}\right)
   \end{equation*}
   which tends to $0$ as $\epsilon\downarrow 0$. In this way, we recover uniform convergence.

\end{proof}

The role of the assumption $k(n)\ll n$ is that the training period is negligible compared with the time scale $n$ on which the return is observed. In the proof, this assumption enters when we estimate the denominator in the
Azuma-Hoeffding bound. If, instead, we have $k(n)\sim n$, then the training would occupy a positive fraction of the memory at the observation scale, and the present argument would no longer apply without a separate analysis.

We are now ready to examine the behavior of the first return time of the trained ERW, denoted as $T^{(k)}$,  which captures the speed at which the elephant forgets its training. 

\begin{proof}[Proof of Theorem \ref{first return time for original ERW}]
Recall the definition of the first return time of the trained ERW in (\ref{definition of return time diffusive})
    \begin{align*}
        T^{(k)}=\inf\{n\geq k: S^{(k)}(n)=0\}
    \end{align*}
With the reparametrization suggested by (\ref{critical training phase}), we can rewrite it as:
\begin{equation*}
    T^{(k)}=k+(3-4p)^{\frac{1}{3-4p}}k^{\frac{4-4p}{3-4p}}\inf\left\{t\geq 0: S^{(k)}\left(k+\left\lfloor{(3-4p)^{\frac{1}{3-4p}}k^{\frac{4-4p}{3-4p}}t}\right\rfloor\right)\leq 0\right\}
\end{equation*}
Set
\begin{equation*}
r(k):=(3-4p)^{1/(3-4p)}\,k^{(4-4p)/(3-4p)}
\end{equation*}
and define
\begin{equation*}
Z^{(k)}(t):=\frac{S^{(k)}\left(k+\lfloor{r(k)t}\rfloor\right)}{\sqrt{r(k)}}, \quad t\geq 0
\end{equation*}
Then
\begin{equation*}
\frac{T^{(k)}}{r(k)}=\frac{k}{r(k)}+\inf\left\{t\geq 0: Z^{(k)}(t)\leq 0\right\}
\end{equation*}
Since $r(k)\ll k$, it is enough to identify the limit of the hitting time on the right-hand side. 

For every $0<\epsilon<\ell<\infty$, we recall that Proposition \ref{functional convergence for rescaled and conditioned ERW} gives
\begin{equation*}
\left(Z^{(k)}(t)\right)_{t\in[\epsilon,\ell]}\xrightarrow[k\rightarrow\infty]{(d)}\left(Z(t)\right)_{t\in[\epsilon,\ell]}\quad\text{in }D([\epsilon,\ell])
\end{equation*}
where
\begin{equation*}
Z(t):=\hat B(t)+\frac{t^{2p-1}}{\sqrt{3-4p}}
\end{equation*}
Using the usual representation of the noise-reinforced Brownian motion, we may write
\begin{equation*}
Z(t)=\frac{t^{2p-1}}{\sqrt{3-4p}} \left(B\left(t^{3-4p}\right)+1\right),\quad t>0
\end{equation*}
for a standard Brownian motion $B$. Hence
\begin{equation}\label{return time for Z}
\inf\left\{t\geq 0: Z(t)\leq 0\right\}=\inf\left\{t\geq 0: B\left(t^{3-4p}\right)+1\leq 0\right\}=\tau^{1/(3-4p)}
\end{equation}
where
\begin{equation*}
\tau:=\inf\left\{u\geq 0: B(u)+1=0\right\}>0, \text{\, almost surely}
\end{equation*}
because of the continuity property of Brownian motion.

Now define
\begin{equation*}
H(f):=\inf\left\{t\geq 0:f(t)\leq 0\right\}, \quad H_{\epsilon}(f):=\inf\left\{t\geq \epsilon:f(t)\leq 0\right\}
\end{equation*}
For fixed $\epsilon>0$ and $\ell>\epsilon$, we pass from the convergence of
the processes to the convergence of the corresponding hitting times by applying
the continuous mapping theorem to the functional
\begin{equation*}
f\longmapsto H_{\epsilon}(f)\wedge \ell
\end{equation*}
It remains to verify that this functional is almost surely continuous at the limiting process $Z$.

Since the limiting process $Z$ has continuous paths, it is enough to verify that, whenever its first entrance into $(-\infty,0]$ occurs before time $\ell$, the path crosses the level $0$ rather than merely touching it. By the Brownian representation
\begin{equation*}
    Z(t)=\frac{t^{2p-1}}{\sqrt{3-4p}}\left(B\left(t^{3-4p}\right)+1\right)
\end{equation*}
and since the prefactor is strictly positive, zeros of $Z$ correspond to hits of the level $-1$ by Brownian motion. Moreover, $\mathbb P(Z(\epsilon)=0)=0$. On the event $\{Z(\epsilon)>0,\ H_\epsilon(Z)<\ell\}$, the strong Markov property of Brownian motion at the corresponding hitting time of the level $-1$ implies that the shifted process after this time is a Brownian motion starting from $0$. It is well known that such a Brownian motion enters the negative half-line immediately almost surely. Therefore $Z$ crosses below $0$ immediately after its first entrance into $(-\infty,0]$. On the event $\{Z(\epsilon)<0\}$, we have $H_\epsilon(Z)=\epsilon$, so continuity is immediate. If $H_\epsilon(Z)\geq \ell$, the truncation by $\ell$ also makes the continuity immediate. Hence, $f\mapsto H_\epsilon(f)\wedge\ell$ is almost surely continuous at $Z$. The continuous mapping theorem gives 
\begin{equation*}
H_{\epsilon}(Z^{(k)})\wedge \ell \xrightarrow[k\rightarrow\infty]{(d)} H_{\epsilon}(Z)\wedge \ell
\end{equation*}
Since $H_{\epsilon}(Z)<\infty$ almost surely, letting $\ell\to\infty$ yields
\begin{equation}\label{H_epsilon(Z^{(k)}) converges to H_{epsilon}(Z)}
H_\epsilon(Z^{(k)})\xrightarrow[k\rightarrow\infty]{(d)} H_{\epsilon}(Z)\quad\text{for every }\epsilon>0
\end{equation}
It remains to let $\epsilon\downarrow 0$. By Lemma \ref{unlikely to return zero}, it gives us
\begin{equation}\label{equiv of lemma 3.1}
\lim_{\epsilon\downarrow 0}\limsup_{k\to\infty}\mathbb{P}\left(H(Z^{(k)})\leq \epsilon\right)=0
\end{equation}
On the other hand, by (\ref{return time for Z})
\begin{equation*}
H(Z)=\tau^{1/(3-4p)}>0
\qquad\text{almost surely}
\end{equation*}
therefore,
\begin{equation}\label{H_{epsilon}(Z) converges to H(Z)}
H_{\epsilon}(Z)\rightarrow H(Z)
\qquad\text{almost surely as }\epsilon\rightarrow 0
\end{equation}
Let $\varphi:\mathbb{R}_+\to\mathbb{R}$ be any bounded and continuous function. Then, for every $\epsilon>0$, by triangle inequality 
\begin{align}\label{Z^{(k)} converges to Z}
\left|\mathbb{E}\left(\varphi\left(H(Z^{(k)})\right)\right)-
\mathbb{E}\left(\varphi\left(H(Z)\right)\right)\right|
&\leq2\|\varphi\|_\infty\mathbb{P}\left(H(Z^{(k)})\leq \epsilon\right) \nonumber\\
&+\left|\mathbb{E}\left(\varphi\left(H_\epsilon(Z^{(k)})\right)\right)
-\mathbb{E}\left(\varphi\left(H_\epsilon(Z)\right)\right)
\right| \nonumber\\
&+\mathbb{E}\left(\left|\varphi\left(H_\epsilon(Z)\right)-\varphi\left(H(Z)\right)\right|\right)
\end{align}
Letting first $k\to\infty$, and then $\epsilon\downarrow 0$, we see that the quantity in (\ref{Z^{(k)} converges to Z}) tends to $0$. The first term in the upper bound goes to $0$ due to (\ref{equiv of lemma 3.1}), while the second and the third term goes to $0$ by (\ref{H_epsilon(Z^{(k)}) converges to H_{epsilon}(Z)}) and (\ref{H_{epsilon}(Z) converges to H(Z)}), respectively. We obtain
\begin{equation*}
H(Z^{(k)})\xrightarrow[k\rightarrow\infty]{(d)}H(Z)=\tau^{1/(3-4p)}
\end{equation*}
Therefore,
\begin{equation*}
\frac{T^{(k)}}{r(k)}=\frac{k}{r(k)}+H(Z^{(k)})\xrightarrow[k\rightarrow\infty]{(d)}\tau^{1/(3-4p)}
\end{equation*}
Equivalently,
\begin{equation*}
k^{-(4-4p)/(3-4p)}T^{(k)}\xrightarrow[k\rightarrow\infty]{(d)}(3-4p)^{1/(3-4p)}\,\tau^{1/(3-4p)}
\end{equation*}
which proves the theorem.
\end{proof}

For the case where the initial training brings the elephant to a position of order $ck$ at time $k$, with $c \in (0,1)$, the analysis proceeds similarly to the case discussed above and is therefore omitted.

\section{The trained ERW in the critical regime}\label{critical regime}

We now turn to the critical regime, where the memory parameter is fixed at $p=3/4$. We present results without proof when they mirror those established in the diffusive case. As before, our focus remains on the long-term behavior of the trained process $(S^{(k)}(n))_{n\geq 0}$, specifically its asymptotic behavior when $n \gg k$.

The martingale construction is the same in spirit. In this context, we introduce a sequence of scaling factors defined for every $j\in\mathbb{N}$,
\begin{equation*}
    a_{j}:=\frac{\Gamma\left(j \right)}{\Gamma \left(j+1/2 \right)}
\end{equation*}

Moreover, applying Stirling’s formula and using the fact that $p=3/4$, we obtain the following asymptotic behavior for every $t> 0$
\begin{equation*}\label{asymp coeff critical}
    a_{j}\sim j^{-1/2} \text{\; as $j\rightarrow\infty$} \quad \text{together with} \quad
\sum_{1\leq j\leq \lfloor{n^t}\rfloor} a_j^2 \sim t\log n \text{\; as $n\rightarrow\infty$}
\end{equation*}
As a consequence, the natural time reparametrization becomes $\lfloor{n^t}\rfloor$ rather than $\lfloor{nt}\rfloor$. Then we define the process for $t\geq 0$
\begin{equation*}
    \left(M_2^{(k,n)}(t)\right)_{t\geq 0}:= \left(a_{\lfloor{k(n)}\rfloor+\lfloor{n^t}\rfloor}S^{(k(n))}\left(\lfloor{k(n)}\rfloor+\lfloor{n^t}\rfloor\right)-a_{\lfloor{k(n)}\rfloor+1}S^{(k(n))}\left(\lfloor{k(n)}\rfloor+1\right)\right)_{t\geq 0}
\end{equation*}

Similarly to the diffusive regime, we are interested in the long-term behavior of the above process. The present reparameterization suggests that we assume $k(n)\ll n^t$ for every $t>0$, which is equivalent to saying $\log k(n)\ll \log n$ as $n\to\infty$.

We next state without proof a result analogous to Lemma \ref{diffusive martingale} in the present context.
\begin{lemma}\label{critical martingale}
    We have the next properties for the process $(M_2^{(k,n)}(t))_{t\geq 0}$,
    \begin{enumerate}
        \item For every $n\in\mathbb{N}$, $(M_2^{(k,n)}(t))_{t\geq 0}$ is a $\mathbb{P}_{k(n)}$-martingale that starts from 0.
        \item For every $t\geq 0$, the quadratic variation of $(M_2^{(k,n)}(t))_{t\geq 0}$ behaves asymptotically as: 
        \begin{equation*}\label{critical quadratic variation 2}
            \left[M_2^{(k,n)}, M_2^{(k,n)}\right](t) \sim t\log n \quad \text{as $n\rightarrow\infty$}
        \end{equation*}
    \end{enumerate}
\end{lemma}

The martingale property in Lemma \ref{critical martingale} is obtained in the same way as before. The essential difference lies in the quadratic variation: in the diffusive regime, one sums terms of order $j^{2-4p}$, which produces a power of $n$, whereas at criticality one has $a_j^2\sim j^{-1}$. Together with the assumption $\log k(n)\ll \log n$, the corresponding sum grows like $\log n$.

It is worth pointing out that by using the same martingale functional central limit theorem (FCLT) approach, we find that the critical training phase now becomes $\log n$. Moreover, the first return time becomes exponential rather than polynomial in the training length. Thus, the arguments below are parallel to those of Sections \ref{section:preliminaries} and \ref{main proof}, but the power law estimates used there must be replaced throughout by logarithmic ones.

\begin{proposition}\label{functional convergence for rescaled and conditioned ERW 2}
    In the critical regime, $p=3/4$, if $k(n)\sim \log n$ as $n\rightarrow\infty$, then the following weak convergence holds in the Skorokhod topology: for each pair of $0<\epsilon\leq l <\infty$,
\begin{equation}\label{ERW scaling limit diffusive 2}
    \left(\frac{S^{(k(n))}\left(\lfloor{n^t}\rfloor\right)}{\sqrt{n^t\log n}}\right)_{t\in[\epsilon,l] }\xrightarrow[n\rightarrow\infty]{(d)} \left(B(t)+1\right)_{t\in[\epsilon,l] }
\end{equation}
where $(B(t))_{t\geq 0}$ is the standard Brownian motion. We emphasize that the convergence in (\ref{ERW scaling limit diffusive 2}) is valid for any positive compact interval bounded away from $0$.
\end{proposition}

If $k(n)\ll \log n$, i.e., the initial training is negligible compared to the critical training phase in the critical regime. In that case, the deterministic effect of the training becomes negligible compared with the leading normalization $\sqrt{n^t\log n}$. As a result, the limiting process in (\ref{ERW scaling limit diffusive 2}) is simply a standard Brownian motion.

As in the diffusive case, we now study the first return time of the trained ERW in the critical regime. Recall that the first return time of the trained ERW in the critical regime is defined as, for every $k\geq 0$:
\begin{align*}
T^{(k)} := \inf\{n \geq k : S^{(k)}(n) = 0\}
\end{align*}
Due to the time reparametrization introduced in Proposition \ref{functional convergence for rescaled and conditioned ERW 2}, we can equivalently express it as
\begin{align*}
T^{(k(n))} = \lfloor{k(n)}\rfloor + \lfloor{n^ {\inf\left\{ t\geq 0 : S^{(k(n))}\left(\lfloor{k(n)}\rfloor + \lfloor{n^t}\rfloor\right) = 0 \right\}}}\rfloor
\end{align*}

To pass from Proposition \ref{functional convergence for rescaled and conditioned ERW 2} to Theorem \ref{first return time for original ERW 2}, one needs the critical counterpart of Lemma \ref{unlikely to return zero}. The argument is based on the same Azuma--Hoeffding localization as in the diffusive regime, but the relevant estimates are again logarithmic rather than polynomial. More precisely, when $k(n)\sim \log n$, the initial weighted displacement $a_{\lfloor{k(n)}\rfloor}\lfloor{k(n)}\rfloor$ is of order $\sqrt{\log n}$, while the quadratic variation accumulated up to time $\lfloor{n^{\epsilon}}\rfloor$ is of order $\epsilon\log n$. Hence, the probability of a return on a sub-exponential time scale is uniformly small as $\epsilon\downarrow 0$. Once this critical no-fast-return estimate is established, the stopping-time argument used in the proof of Theorem \ref{first return time for original ERW} carries over without further change.

\begin{theorem} \label{first return time for original ERW 2}
In the critical regime, $p=3/4$, we have 
\begin{equation}\label{intro: return time diffusive 2}
   \frac{\log T^{(k)}}{k} \xrightarrow[k\rightarrow\infty]{(d)} \tau\
\end{equation}
where $\tau$ is a $\text{Stable}(1/2)$-distributed random variable with density 
\begin{equation*}
    \frac{1}{\sqrt{2\pi t^{3}}}\exp{\left(-\frac{1}{2t}\right)} \quad \text{for $t>0$.}
\end{equation*} 
\end{theorem}

We now turn to the case where the elephant is overly trained, namely, the initial training satisfies $k(n)\gg \log n$, while still ensuring that $\log k(n)\ll \log n$. The next result mirrors Proposition \ref{convergence for k(n)>> critical training phase}.

\begin{proposition}\label{convergence for k(n)>> critical training phase 2}
    In the critical regime, $p=3/4$, if $k(n) \gg \log n$ as $n\rightarrow\infty$, we have the following weak convergence
    \begin{align*}
        \frac{S^{(k(n))}\left(n\right)-\sqrt{nk(n)}}{\sqrt{n\log n}}\xrightarrow[n\rightarrow\infty]{(d)}\mathcal{N}\left(0\, , 1\right)
    \end{align*}
\end{proposition}

As in the diffusive regime, we leave the general case to the reader, since it follows closely from the results presented above.

\section*{Acknowledgement} 

I would like to thank the two anonymous referees for their careful reading of the manuscript and for their valuable suggestions. Their detailed comments helped improve the exposition, in particular its conceptual motivation, the discussion of its relation to persistence exponents, and the clarity of several proofs. I am also deeply grateful to Jean Bertoin for many stimulating discussions and for sharing his insights, which greatly influenced the ideas and results presented in this paper.

\noindent 

\end{document}